\documentclass{article}%
\usepackage{amsfonts}
\usepackage{amssymb}
\usepackage{graphicx}
\usepackage{amsmath}
\usepackage{amscd}%
\newtheorem{theorem}{Theorem}

\newtheorem{corollary}[theorem]{Corollary}

\newtheorem{definition}[theorem]{Definition}

\newtheorem{lemma}[theorem]{Lemma}

\newtheorem{proposition}[theorem]{Proposition}

\newenvironment{proof}[1][Proof]{\textbf{#1.} }{\ \rule{0.5em}{0.5em}}
\begin{document}

\title{Splitting cubic circle graphs}
\author{Lorenzo Traldi\\Lafayette College\\Easton, Pennsylvania 18042}
\date{}
\maketitle

\begin{abstract}
We show that every 3-regular circle graph has at least two pairs of twin
vertices; consequently no such graph is prime with respect to the split
decomposition. We also deduce that up to isomorphism, $K_{4}$ and $K_{3,3}$
are the only 3-connected, 3-regular circle graphs.

\bigskip


\bigskip

Mathematics Subject\ Classification. 05C62

\end{abstract}

\section{Introduction}

Circle graphs have been introduced several times, in several contexts. The
intersection graph of a family of chords in a circle seems to have first been
mentioned in print by Zelinka \cite{Z}; he gave credit for the idea to Kotzig,
whose seminal work \cite{K} founded the special theory of 4-regular graphs.
Brahana's separation matrix \cite{Br} -- in essence, the adjacency matrix of a
circle graph -- was introduced decades earlier, in connection with the
geometry of surfaces. Circle graphs achieved broad recognition in the 1970s,
when Even and Itai \cite{EI} considered circle graphs in relation to the
analysis of permutations using stack and queues; Bouchet \cite{Bold} and Read
and Rosenstiehl \cite{RR} discussed the interlacement graphs of double
occurrence words in connection with the famous Gauss problem of characterizing
generic curves in the plane; and Cohn and Lempel \cite{CL} related the cycle
structure of a certain kind of permutation to the $GF(2)$-nullity of an
associated link relation matrix (which is also the adjacency matrix of a
circle graph).\ An account of the early combinatorial theory appears in
Golumbic's classic book \cite{Go}.

\begin{definition}
Let $W=w_{1}...w_{2n}$ be a double occurrence word, i.e., a sequence in which
$n$ letters appear, each letter appearing twice. Then the \emph{interlacement
graph} $\mathcal{I}(W)$ is a graph with $n$ vertices, labeled by the letters
appearing in $W$. Two vertices $a$ and $b$ of $\mathcal{I}(W)$ are adjacent if
and only if the corresponding letters appear in $W$ in the order $abab$ or
$baba$. A simple graph that can be realized as an interlacement graph of a
double occurrence word is a \emph{circle graph}.
\end{definition}

During the last forty years the theory of circle graphs has been sharpened
considerably. Polynomial-time recognition algorithms were developed before the
new millennium by Bouchet \cite{Bec}, Naji \cite{N} and Spinrad \cite{Sp}.
More recently, Courcelle \cite{C} has observed that circle graphs are well
described in the framework of monadic second-order logic, and Gioan, Paul,
Tedder and Corneil have provided the first subquadratic recognition algorithm
\cite{GPTC, GPTC1}.

The crucial tool used to design recognition algorithms for circle graphs is
the split decomposition of Cunningham \cite{Cu}. We recall only the basic
definition here, and defer to the literature (\cite{C, Cu, GPTC1} for
instance) for thorough explanations of this important idea.

\begin{definition}
Let $G$ be a simple graph. A \emph{split} $(V_{1}$, $X_{1}$; $V_{2}$, $X_{2})$
of $G$ is given by a partition $V(G)=V_{1}\cup V_{2}$ with $\left\vert
V_{1}\right\vert $, $\left\vert V_{2}\right\vert \geq2$ and subsets
$X_{i}\subseteq V_{i}$ with these properties: the complete bipartite graph
with vertex-classes $X_{1}$ and $X_{2}$ is a subgraph of $G$, and $G$ does not
have any other edge from $V_{1}$ to $V_{2}$.
\end{definition}

Every simple graph of order 4 has a split. A graph with five or more vertices
that has no split is said to be \emph{prime}.

Three types of splits are particularly simple. Let $G$ be a graph with four or
more vertices.

\begin{itemize}
\item Suppose $v_{1}$ and $v_{2}$ are twin vertices of $G$, i.e., they have
the same neighbors outside $\{v_{1},v_{2}\}$. Then $G$ has a split with
$V_{1}=X_{1}=\{v_{1},v_{2}\}$.

\item Suppose $G$ is not connected. Let $H$ be a union of some but not all
connected components of $G$, such that $H$ includes at least two vertices. If
$h\in V(H)$ then $G$ has a split with $V_{1}=V(H)$ and $X_{1}=\varnothing$ or
$G$ has a split with $V_{1}=V(H-h)$ and $X_{2}=\{h\}$.

\item Suppose $G$ has a cutpoint $v$. Let $H$ be a union of some but not all
connected components of $G-v$, such that $H$ includes at least two vertices.
Then $G$ has a split with $V_{1}=V(H)$ and $X_{2}=\{v\}$.
\end{itemize}

A fundamental part of the theory of circle graphs and split decompositions is
the following operation, which is motivated by the properties of double
occurrence words \cite{Bold, K, RR}. We use $N(v)$ to denote the open
neighborhood of $v$ in $G$, i.e., the set of vertices $w\neq v$ such that
$vw\in E(G)$.

\begin{definition}
Let $v$ be a vertex of a simple graph $G$. Then the \emph{local complement }of
$G$ with respect to $v$ is the graph $G^{v}$ with $V(G^{v})=V(G)$ and
$E(G^{v})=\{xy\mid$ either $x\notin N(v)$ and $xy\in E(G)$ or $x,y\in N(v)$
and $xy\notin E(G)\}$.
\end{definition}

In some references this operation is called \emph{simple} local
complementation, to distinguish it from a related operation that involves
looped vertices. We consider only simple graphs in this paper, so we need not
be so careful here. Two important properties of local complementation are that
$G$ is a circle graph if and only if $G^{v}$ is a circle graph, and that $G$
has a split $(V_{1}$, $X_{1}$; $V_{2}$, $X_{2})$ if and only if $G^{v}$ has a
split $(V_{1}$, $Y_{1}$; $V_{2}$, $Y_{2})$.

A central result involving circle graphs is Bouchet's characterization by
obstructions \cite{Bco}. Recall that \emph{local equivalence} is the
equivalence relation generated by isomorphisms and local complementations.

\begin{theorem}
\cite{Bco} A simple graph $G$ is a circle graph if and only if no graph
locally equivalent to $G$ has one of the graphs of Figure \ref{ccircf1} as an
induced subgraph.
\end{theorem}

%

%TCIMACRO{\FRAME{ftbpFU}{4.0949in}{1.0101in}{0pt}{\Qcb{Bouchet's obstructions:
%$W_{5}$, $BW_{3}$ and $W_{7}$.}}{\Qlb{ccircf1}}{ccircf1.ps}%
%width 4.0949in;  height 1.0101in;  depth 0pt;  original-width 8.4968in;
%original-height 11.0056in;  cropleft "0.1883";  croptop "0.8420";
%cropright "0.8262";  cropbottom "0.7231";
%filename 'ccircf1.ps';file-properties "XNPEU";}}}%
%BeginExpansion
\begin{figure}
[ptb]
\begin{center}
\includegraphics[
trim=1.599947in 7.958149in 1.476744in 1.738885in,
height=1.0101in,
width=4.0949in
]%
{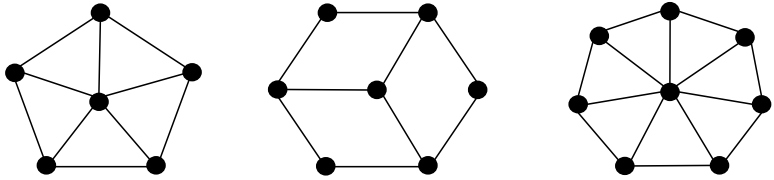}%
\caption{Bouchet's obstructions: $W_{5}$, $BW_{3}$ and $W_{7}$.}%
\label{ccircf1}%
\end{center}
\end{figure}
%EndExpansion

Observe that $W_{5}$ and $W_{7}$ are locally equivalent to the graphs pictured
in Figure \ref{ccircf14}, both of which are prime, 3-connected, and 3-regular.
The purpose of this paper is to prove that these local equivalences are no
mere coincidence:

\begin{theorem}
\label{sharp}Let $G$ be a 3-regular circle graph. Then $G$ has at least two
disjoint pairs of twin vertices.
\end{theorem}

%

%TCIMACRO{\FRAME{ftbpFU}{2.7864in}{1.049in}{0pt}{\Qcb{Prime, 3-connected cubic
%graphs locally equivalent to $W_{5}$ and $W_{7}$.}}{\Qlb{ccircf14}%
%}{ccircf14.ps}{\special{ language "Scientific Word";  type "GRAPHIC";
%width 2.7864in;  height 1.049in;  depth 0pt;  original-width 8.4968in;
%original-height 11.0056in;  cropleft "0.3462";  croptop "0.9148";
%cropright "0.7791";  cropbottom "0.7912";
%filename 'ccircf14.ps';file-properties "XNPEU";}}}%
%BeginExpansion
\begin{figure}
[ptb]
\begin{center}
\includegraphics[
trim=2.941592in 8.707630in 1.876943in 0.937677in,
height=1.049in,
width=2.7864in
]%
{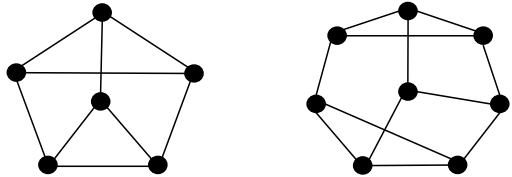}%
\caption{Prime, 3-connected cubic graphs locally equivalent to $W_{5}$ and
$W_{7}$.}%
\label{ccircf14}%
\end{center}
\end{figure}
%EndExpansion

Theorem \ref{sharp} immediately implies the following.

\begin{corollary}
Let $G$ be a 3-regular circle graph. Then $G$ is not prime.
\end{corollary}

In\ Section 4 we deduce another consequence of Theorem \ref{sharp}.

\begin{corollary}
\label{connect}Let $G$ be a 3-regular circle graph, which is not isomorphic to
$K_{4}$ or $K_{3,3}$. Then $G$ is not 3-connected.
\end{corollary}

Before proceeding we should take a moment to thank Robert Brijder for the many
inspirations provided by our long correspondence and collaboration.

\section{Three lemmas}

In this section we recall three elementary results about double occurrence
words and circle graphs.

\begin{lemma}
If $G$ is a circle graph with a vertex $v$ then $G-v$ and $G^{v}$\ are also
circle graphs.
\end{lemma}

\begin{proof}
For $G-v$, take a double occurrence word whose interlacement graph is $G$ and
remove the two occurrences of $v$. For $G^{v}$, take a double occurrence word
whose interlacement graph is $G$ and reverse the subword between the two
occurrences of $v$.
\end{proof}

\begin{lemma}
\label{cyclice}If $W^{\prime}$ is obtained from a double occurrence word $W$
by some sequence of cyclic permutations
\[
w_{1}...w_{2n}\mapsto w_{i}w_{i+1}...w_{2n}w_{1}..w_{i-1}%
\]
and reversals%
\[
w_{1}...w_{2n}\mapsto w_{2n}w_{2n-1}...w_{2}w_{1}%
\]
then $\mathcal{I}(W)=\mathcal{I}(W^{\prime})$.
\end{lemma}

Double occurrence words related as in Lemma \ref{cyclice} are said to be
\emph{cyclically equivalent}. By the way, the converse of Lemma \ref{cyclice}
is false in general; a complete characterization of double occurrence words
with the same connected interlacement graph has been provided by Ghier
\cite{Gh}.

\begin{lemma}
\label{easy}Let $G_{1}$ and $G_{2}$ be disjoint simple graphs, and let $G$ be
a graph obtained by attaching $G_{1}$ to $G_{2}$ with a single edge. Then $G$
is a circle graph if and only if both $G_{1}$ and $G_{2}$ are circle graphs.
\end{lemma}

\begin{proof}
If $G$ is a circle graph, it yields $G_{1}$ and $G_{2}$ through vertex deletion.

For the converse, suppose $W_{1}=w_{1}...w_{2a}$ and $W_{2}=x_{1}...x_{_{2b}}$
are double occurrence words whose interlacement graphs are $G_{1}$ and $G_{2}%
$, respectively. After cyclic permutation, we may presume that the one
additional edge of $G$ attaches $w_{2a}$ to $x_{1}$. Then the word
\[
W=w_{1}...w_{2a-1}x_{1}w_{2a}x_{2}...x_{2b}%
\]
has $G$ as its interlacement graph.
\end{proof}

\section{Proof of Theorem \ref{sharp}}

Before beginning the proof, we should mention that we sometimes say
\textquotedblleft two pairs\textquotedblright\ rather than \textquotedblleft
two disjoint pairs\textquotedblright\ while discussing Theorem \ref{sharp}. In
fact, the theorem is equivalent to the weaker-seeming assertion that every
cubic circle graph has two pairs of twin vertices. For if a cubic graph has
two intersecting pairs of adjacent twins, then the three twin vertices and
their shared neighbor constitute a 4-clique, which must be a whole connected
component; the 4-clique provides three disjoint pairs of twin vertices. If a
cubic graph has two intersecting pairs of nonadjacent twins then the three
twin vertices and their three neighbors constitute a 3,3-biclique, which must
again be a whole connected component; the biclique provides nine disjoint
pairs of twin vertices. (An easy argument shows that a pair of adjacent twins
cannot intersect a pair of nonadjacent twins in any graph.)

\subsection{A minimal counterexample is 2-connected}

Let $G$ be a minimal counterexample to Theorem \ref{sharp}, i.e., a cubic
circle graph that does not have two disjoint pairs of twin vertices, with the
smallest possible number of vertices. Then $G$ must certainly be connected,
for if not then each connected component of $G$ is itself a smaller counterexample.

\begin{proposition}
\label{cut0}Let $G$ be a cubic circle graph that does not have two pairs of
twin vertices, and is of the smallest order for such a graph. If $G$ is not
2-connected then $G$ has precisely two cutpoints.
\end{proposition}

%

%TCIMACRO{\FRAME{ftbpFU}{4.6925in}{0.6426in}{0pt}{\Qcb{If $G$ has more than two
%cutpoints, its minimality is contradicted.}}{\Qlb{ccircf3a}}{ccircf3a.ps}%
%width 4.6925in;  height 0.6426in;  depth 0pt;  original-width 8.4968in;
%original-height 11.0056in;  cropleft "0.1732";  croptop "0.8902";
%cropright "0.9051";  cropbottom "0.8157";
%filename 'ccircf3a.ps';file-properties "XNPEU";}}}%
%BeginExpansion
\begin{figure}
[ptb]
\begin{center}
\includegraphics[
trim=1.471646in 8.977268in 0.806346in 1.208415in,
height=0.6426in,
width=4.6925in
]%
{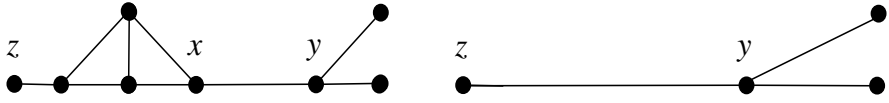}%
\caption{If $G$ has more than two cutpoints, its minimality is contradicted.}%
\label{ccircf3a}%
\end{center}
\end{figure}
%EndExpansion

\begin{proof}
Suppose $G$ has a cutpoint, $x$. As $x$ is of degree 3, one of the components
of $G-x$ is connected to $x$ by only one edge, $e$. Then $e$ is an isthmus,
and its other end-vertex is a cutpoint; denote the other end-vertex $y$.

Suppose $G$ has a third cutpoint, $z$. Interchanging the labels of $x$ and $y$
if necessary, we may presume that $x$ and $z$ are vertices of the same
component of $G-e$; and interchanging the labels of $z$ and a neighbor of $z$,
we may presume that there is an isthmus between $x$ and $z$. (A portion of an
example is indicated on the left in\ Figure \ref{ccircf3a}.) Let $G^{\prime}$
be the cubic graph obtained from $G$ by removing all vertices between $y$ and
$z$, and then attaching $y$ to $z$ by an edge, as indicated on the right in
Figure \ref{ccircf3a}. Lemma \ref{easy} tells us that $G^{\prime}$ is a circle
graph. Moreover, it is clear that every pair of twins in $G^{\prime}$ is also
a pair of twins in $G$, so $G^{\prime}$ does not have two pairs of twin
vertices. But this contradicts the minimality of $G$.
\end{proof}

\begin{proposition}
\label{cut1}Let $G$ be a cubic circle graph that does not have two pairs of
twin vertices, and is of the smallest order for such a graph. Suppose $G$ has
two cutpoints, $x$ and $y$. Then at least one of $x$, $y$ does not appear on a
3-circuit of $G$.
\end{proposition}

\begin{proof}
Suppose instead that $x$ and $y$ both appear on 3-circuits of $G$. If the two
other vertices of the 3-circuit containing $x$ are twins then their other
neighbor is a cutpoint, as illustrated on the left in Figure \ref{ccircf3a}.
But $G$ does not have a third cutpoint, so these two vertices must not be
twins. The same argument shows that the other two vertices of the 3-circuit
containing $y$ are not twins, so the situation in $G$ is as indicated on the
left-hand side of Figure \ref{ccircf3b}. As indicated on the right-hand side
of the figure, we obtain two smaller cubic circle graphs by deleting $x$ and
$y$ and then performing local complementations and deletions at the four
resulting degree-2 vertices. Call the two smaller graphs $G_{x}$ and $G_{y}$.
The minimality of $G$ implies that each of $G_{x}$ and $G_{y}$ has two pairs
of twin vertices.%
%TCIMACRO{\FRAME{ftbpFU}{4.8525in}{0.6443in}{0pt}{\Qcb{$G$, $G_{x}$ and $G_{y}%
%$.}}{\Qlb{ccircf3b}}{ccircf3b.ps}{\special{ language "Scientific Word";
%type "GRAPHIC";  maintain-aspect-ratio TRUE;  display "USEDEF";
%valid_file "F";  width 4.8525in;  height 0.6443in;  depth 0pt;
%original-width 8.4968in;  original-height 11.0056in;  cropleft "0.1651";
%croptop "0.9024";  cropright "0.9213";  cropbottom "0.8276";
%filename 'ccircf3b.ps';file-properties "XNPEU";}}}%
%BeginExpansion
\begin{figure}
[ptb]
\begin{center}
\includegraphics[
trim=1.402822in 9.108234in 0.668698in 1.074146in,
height=0.6443in,
width=4.8525in
]%
{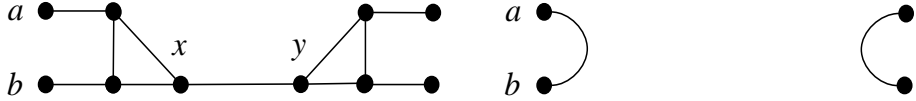}%
\caption{$G$, $G_{x}$ and $G_{y}$.}%
\label{ccircf3b}%
\end{center}
\end{figure}
%EndExpansion

As $G$ does not have two pairs of twin vertices, it must be the case that for
at least one of $G_{x}$ and $G_{y}$, no pair of twin vertices is also a pair
of twins in $G$. We assume that no pair of twin vertices from $G_{x}$ is also
a twin pair in $G$. Then there are two pairs of twins in $G_{x}$ whose twin
relationship is \textquotedblleft disrupted\textquotedblright\ in $G$.

Such disruption can only occur if each pair of twins includes one of the
vertices of $G_{x}$ denoted $a$ and $b$ in Figure \ref{ccircf3b}, because
these are the only vertices of $G_{x}$ whose neighborhoods in $G$ and $G_{x}$
are not the same. Let $a^{\prime}$ and $b^{\prime}$ be vertices of $G$ that
are twins of $a$ and $b$ in $G_{x}$.

Suppose $a$ and $a^{\prime}$ are adjacent twins in $G_{x}$. If $b$ and
$b^{\prime}$ are adjacent too, then $N_{G}(a^{\prime})=\{a,b,b^{\prime}\}$ and
$N_{G}(b^{\prime})=\{a,a^{\prime},b\}$, so $a^{\prime}$ and $b^{\prime}$ are
adjacent twins in $G$. This is a contradiction, so $b$ and $b^{\prime}$ are
not adjacent. Consequently there is a vertex $z$ such that $N_{G_{x}%
}(b)=N_{G_{x}}(b^{\prime})=\{a,a^{\prime},z\}$; but then $z$ is a cutpoint of
$G$, which separates $\{a,a^{\prime},b,b^{\prime}\}$ from the third neighbor
of $z$. We conclude that $a$ and $a^{\prime}$ are nonadjacent twins in $G_{x}%
$; the same argument shows that $b$ and $b^{\prime}$ are nonadjacent twins in
$G_{x}$.

We claim that Figure \ref{ccircf4} accurately reflects the situations in
$G_{x}$ and $G$. To verify the claim, note first that the vertices labeled $c$
and $d$ in\ Figure \ref{ccircf4} must be distinct, as $G$ has no cutpoint
other than $x$ and $y$. These vertices must also be nonadjacent, because their
being adjacent would imply that $d$ is a twin of $a^{\prime}$ and $c$ is a
twin of $b^{\prime}$, contradicting the hypothesis that $G$ does not have two
pairs of twins. Furthermore, $c$ and $d$ cannot share a neighbor, because a
shared neighbor would be a new cutpoint of $G$. These three observations
justify the claim.%
%TCIMACRO{\FRAME{ftbpFU}{4.4936in}{1.1459in}{0pt}{\Qcb{Two pairs of twins in
%$G_{x}$ are not twins in $G$.}}{\Qlb{ccircf4}}{ccircf4.ps}%
%width 4.4936in;  height 1.1459in;  depth 0pt;  original-width 8.4968in;
%original-height 11.0056in;  cropleft "0.2201";  croptop "0.8539";
%cropright "0.9209";  cropbottom "0.7182";
%filename 'ccircf4.ps';file-properties "XNPEU";}}}%
%BeginExpansion
\begin{figure}
[ptb]
\begin{center}
\includegraphics[
trim=1.870146in 7.904222in 0.672097in 1.607918in,
height=1.1459in,
width=4.4936in
]%
{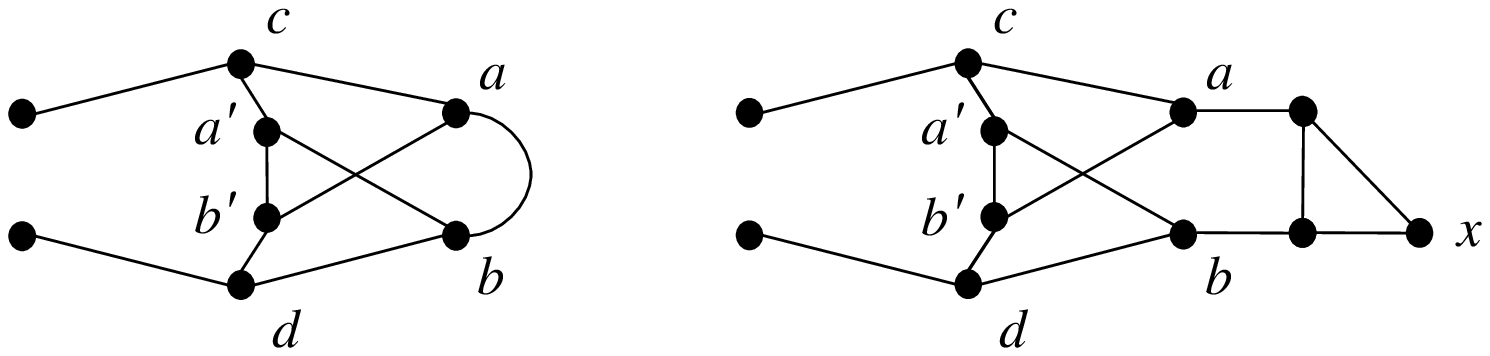}%
\caption{Two pairs of twins in $G_{x}$ are not twins in $G$.}%
\label{ccircf4}%
\end{center}
\end{figure}
%EndExpansion

Now, notice that Figure \ref{ccircf5} indicates that the same sort of twin
disruption occurs in a cubic graph $G^{\prime}$ that is smaller than $G$. More
generally, it is clear that $G$ and $G^{\prime}$ have precisely the same pairs
of twin vertices. The minimality of $G$ requires that $G^{\prime}$ not be a
circle graph.
%TCIMACRO{\FRAME{ftbpFU}{3.5916in}{0.8484in}{0pt}{\Qcb{Two pairs of twins in
%$G_{x}$ are not twins in $G^{\prime}$.}}{\Qlb{ccircf5}}{ccircf5.ps}%
%width 3.5916in;  height 0.8484in;  depth 0pt;  original-width 8.4968in;
%original-height 11.0056in;  cropleft "0.2047";  croptop "0.9025";
%cropright "0.7637";  cropbottom "0.8030";
%filename 'ccircf5.ps';file-properties "XNPEU";}}}%
%BeginExpansion
\begin{figure}
[ptb]
\begin{center}
\includegraphics[
trim=1.739295in 8.837497in 2.007794in 1.073046in,
height=0.8484in,
width=3.5916in
]%
{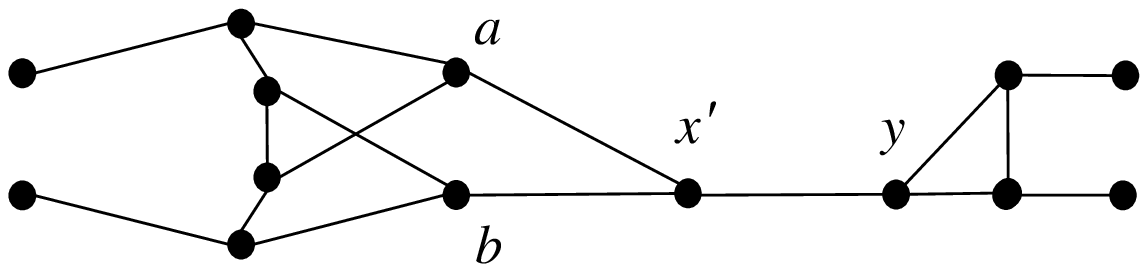}%
\caption{Two pairs of twins in $G_{x}$ are not twins in $G^{\prime}$.}%
\label{ccircf5}%
\end{center}
\end{figure}
%EndExpansion

However, we can obtain $G^{\prime}$ from $G$ as follows. Let $C_{x}$ and
$C_{y}$ be the connected components of $G-xy$. Lemma \ref{easy} tells us that
both are circle graphs. Observe that $C_{x}-x$ has two vertices of degree 2,
the neighbors of $x$ in $C_{x}$. Denote these neighbors $x^{\prime}$ and
$x^{\prime\prime}$, take the local complement of $C_{x}-x$ with respect to
$x^{\prime\prime}$, and then remove $x^{\prime\prime}$. Denote the resulting
graph $C_{x}^{\prime}$. We obtain $G^{\prime}$ from $C_{x}^{\prime}$ and
$C_{y}$ by attaching $x^{\prime}$ to $y$ with an edge, so Lemma \ref{easy}
tells us that $G^{\prime}$ is a circle graph.
\end{proof}

\begin{proposition}
\label{cut2}Let $G$ be a cubic circle graph that does not have two pairs of
twin vertices, and is of the smallest order for such a graph. Then $G$ has no cutpoint.
\end{proposition}

\begin{proof}
Suppose instead that such a $G$ has a cutpoint, $x$. According to Propositions
\ref{cut0} and \ref{cut1}, $x$ has a neighbor $y$ that is the only other
cutpoint of $G$, and we may presume that $x$ does not appear on a 3-circuit of
$G$. Consequently $G$ must fall into one of the two cases pictured on the
left-hand side of Figure \ref{ccircf6}. Let $G_{x}$ and $G_{y}$ be the two
smaller graphs pictured on the right-hand side of Figure \ref{ccircf6}. Each
of $G_{x}$, $G_{y}$ is obtained from $G$ using vertex deletions and local
complementations, so each of $G_{x}$, $G_{y}$ is a circle graph; the
minimality of $G$ implies that each of $G_{x}$, $G_{y}$ has two pairs of twin
vertices.%
%TCIMACRO{\FRAME{ftbpFU}{4.7937in}{1.4468in}{0pt}{\Qcb{The two cases of
%Proposition \ref{cut2}.}}{\Qlb{ccircf6}}{ccircf6.ps}%
%width 4.7937in;  height 1.4468in;  depth 0pt;  original-width 8.4968in;
%original-height 11.0056in;  cropleft "0.1573";  croptop "0.8903";
%cropright "0.9054";  cropbottom "0.7182";
%filename 'ccircf6.ps';file-properties "XNPEU";}}}%
%BeginExpansion
\begin{figure}
[ptb]
\begin{center}
\includegraphics[
trim=1.336547in 7.904222in 0.803797in 1.207315in,
height=1.4468in,
width=4.7937in
]%
{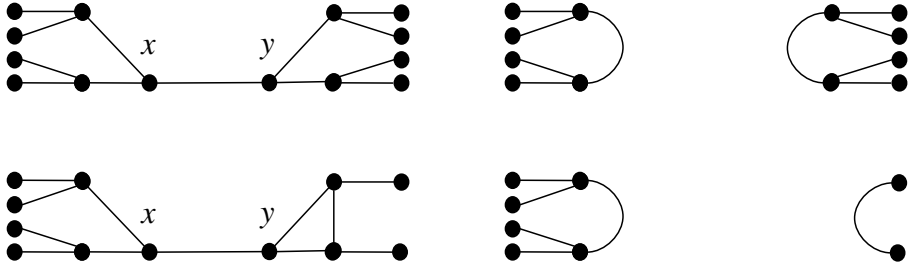}%
\caption{The two cases of Proposition \ref{cut2}.}%
\label{ccircf6}%
\end{center}
\end{figure}
%EndExpansion

Most of the proof consists of a verification of the following.

\textbf{Claim}. It cannot be that no twins of $G_{x}$ are twins in $G$.

Suppose the claim is false. As in the proof of Proposition \ref{cut1}, this
requires that the two neighbors of $x$ in $G_{x}$ appear in two disjoint twin
pairs, which are disrupted in $G$. Let the neighbors be denoted $a$ and $b$,
and suppose $a^{\prime}$ and $b^{\prime}$ are twins of $a$ and $b$ in $G_{x}$.

Suppose $a$ and $a^{\prime}$ are adjacent twins in $G_{x}$. If $b$ and
$b^{\prime}$ are adjacent too, then $N_{G}(a^{\prime})=\{a,b,b^{\prime}\}$ and
$N_{G}(b^{\prime})=\{a,a^{\prime},b\}$, so $a^{\prime}$ and $b^{\prime}$ are
adjacent twins in $G$. This contradiction tells us that $b$ and $b^{\prime}$
are not adjacent, so there is a vertex $z$ such that $N_{G_{x}}(b)=N_{G_{x}%
}(b^{\prime})=\{a,a^{\prime},z\}$. Then $z$ is a cutpoint of $G$, which
separates $\{a,a^{\prime},b,b^{\prime},x\}$ from the rest of $G_{x}$. We see
that $a$ and $a^{\prime}$ cannot be adjacent; the same argument shows that $b$
and $b^{\prime}$ are nonadjacent too.

There are two configurations to consider, indicated in Figure \ref{ccircf7a}.
In both configurations, there must be no twins in the portion of $G_{x}$ that
is not pictured in Figure \ref{ccircf7a}. To see why one of these
configurations must apply, notice that $a^{\prime}$ and $b^{\prime}$ cannot
share a neighbor, because a shared neighbor would be a cutpoint. Consequently
the two unlabeled vertices of Figure \ref{ccircf7a} are indeed distinct. If
these two vertices were adjacent, they would be twins of $a^{\prime}$ and
$b^{\prime}$ in $G$. And if these two vertices were to share a neighbor, then
that neighbor would be a cutpoint.

Consider the first configuration, indicated on the left in Figure
\ref{ccircf7a}. Let $G^{\prime}$ be the graph obtained from $G$ by removing
$a$, $a^{\prime}$, $b$, $b^{\prime}$ and the two unlabeled vertices that
appear in the figure, and then inserting the edges $cx$ and $dx$. Then
$G^{\prime}$ can also be obtained from $G$ as follows: First, remove
$a^{\prime}$ and $b^{\prime}$. Then, perform a local complementation followed
by a vertex deletion at $a$, $b$, and each of the two unlabeled vertices. It
follows that $G^{\prime}$ is a cubic circle graph. The minimality of $G$
requires that $G^{\prime}$ have at least one pair of twin vertices that are
not twins in $G$; this pair can only be $\{c,d\}$. But then the unpictured
third neighbors of $c$ and $d$ are the same, and this vertex is a cutpoint of
$G$.
%TCIMACRO{\FRAME{ftbpFU}{4.8931in}{0.8786in}{0pt}{\Qcb{Two configurations.}%
%}{\Qlb{ccircf7a}}{ccircf7a.ps}{\special{ language "Scientific Word";
%type "GRAPHIC";  maintain-aspect-ratio TRUE;  display "USEDEF";
%valid_file "F";  width 4.8931in;  height 0.8786in;  depth 0pt;
%original-width 8.4968in;  original-height 11.0056in;  cropleft "0.1732";
%croptop "0.8903";  cropright "0.9369";  cropbottom "0.7876";
%filename 'ccircf7a.ps';file-properties "XNPEU";}}}%
%BeginExpansion
\begin{figure}
[ptb]
\begin{center}
\includegraphics[
trim=1.471646in 8.668011in 0.536148in 1.207315in,
height=0.8786in,
width=4.8931in
]%
{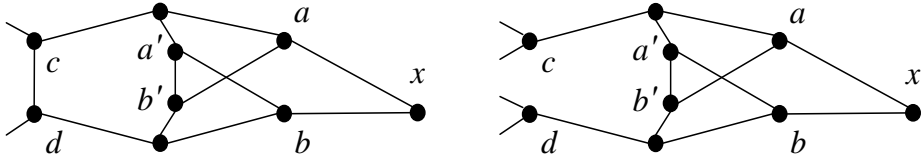}%
\caption{Two configurations.}%
\label{ccircf7a}%
\end{center}
\end{figure}
%EndExpansion%
%TCIMACRO{\FRAME{ftbpFU}{3.7905in}{3.2733in}{0pt}{\Qcb{The graph $G^{\prime}$
%at the bottom has the same twins as $G$.}}{\Qlb{ccircf7}}{ccircf7.ps}%
%width 3.7905in;  height 3.2733in;  depth 0pt;  original-width 8.4968in;
%original-height 11.0056in;  cropleft "0.1564";  croptop "0.7691";
%cropright "0.7468";  cropbottom "0.3763";
%filename 'ccircf7.ps';file-properties "XNPEU";}}}%
%BeginExpansion
\begin{figure}
[ptb]
\begin{center}
\includegraphics[
trim=1.328900in 4.141407in 2.151390in 2.541193in,
height=3.2733in,
width=3.7905in
]%
{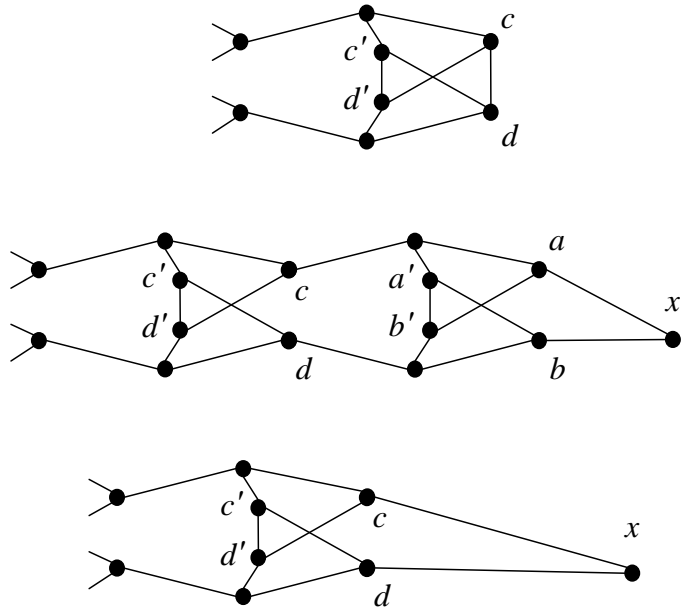}%
\caption{The graph $G^{\prime}$ at the bottom has the same twins as $G$.}%
\label{ccircf7}%
\end{center}
\end{figure}
%EndExpansion

The second configuration resembles the first, but $c$ and $d$ are not
neighbors. Consider the smaller graph obtained from $G_{x}$ by first deleting
the vertices\ $a^{\prime}$ and $b^{\prime}$, then taking the local complement
with respect to each of the four resulting degree-2 vertices in turn, and
deleting it. This smaller graph must have two pairs of twin vertices, and both
of these twin pairs must be disrupted in $G$. This can only occur if each pair
of twins includes one of $c,d$. That is, the smaller graph must be as pictured
at the top of Figure \ref{ccircf7}. The other information in Figure
\ref{ccircf7} follows from our hypotheses: the unlabeled neighbors of
$c^{\prime}$ and $d^{\prime}$ must be distinct (as $G$ has no cutpoint other
than $x$ and $y$); they must be nonadjacent (otherwise they would be twins of
$c^{\prime}$ and $d^{\prime}$); and they must not share a neighbor (as $G$ has
no cutpoint other than $x$ and $y$).

But then $G$ has the structure indicated in the second row of Figure
\ref{ccircf7}, and clearly the graph $G^{\prime}$ indicated at the bottom of
the figure has the same twin vertices $G$ has. This graph is a circle graph,
because it can be obtained from $G$ as follows. First, delete the
vertices\ $a^{\prime}$ and $b^{\prime}$. There are now two unlabeled vertices
of degree 2; one of them has $a$ and $c$ as neighbors and the other has $b$
and $d$ as neighbors. Take the local complements with respect to these two
unlabeled vertices, and then delete them. Now, take the local complements with
respect to $a$ and $b$; after that, delete $a$ and $b$.

We have verified our claim: if no twins of $G_{x}$ are twins in $G$, then the
minimality of $G$ is contradicted. As $G$ does not have two pairs of twins,
the claim implies that no twins of $G_{y}$ are twins in $G$. But if $y$ is not
incident on a 3-circuit then the argument just given contradicts the
minimality of $G$, and if $y$ is incident on a 3-circuit then the argument of
Proposition \ref{cut1} contradicts the minimality of $G$.
\end{proof}

\subsection{A minimal counterexample is 3-connected}

Suppose $G$ is a cubic circle graph that does not have two pairs of twin
vertices, and is of the smallest order for such a graph. We have seen that $G$
must be 2-connected.

Suppose $x\neq y$ and $\{x,y\}$ is a minimal vertex cut of $G$. Each of $x,y$
is of degree 3 in $G$, so for each of them there is a component of $G-x-y$
connected to that vertex by only one edge.

Case 1. Suppose $x$ and $y$ are neighbors. Each then has only one neighbor in
each component of $G-x-y$. If the neighbors in one component are adjacent to
each other, then they form a vertex cut with the same properties as $\{x,y\}$.
We may assume that those two were originally labeled $x$ and $y$, and repeat
this relabeling process as many times as possible. We may do the same thing
with respect to the other component of $G-x-y$, ultimately obtaining the
picture of $G$ indicated on the left-hand side of Figure \ref{ccircf12}, in
which neither $a$ and $b$ nor $a^{\prime}$ and $b^{\prime}$ are neighbors. Let
$H$ and $H^{\prime}$ be the smaller graphs obtained from $G$ as indicated on
the right-hand side of the figure. Then $H$ and $H^{\prime}$ are both cubic
graphs. Moreover, both are circle graphs, as they can be obtained from $G$
using local complementations and vertex deletions; to obtain $H$, for
instance, we delete all vertices outside $C$ except for $x$ and $y$, perform
local complementations at $x$ and $y$, and then delete $x$ and $y$.%

%TCIMACRO{\FRAME{ftbpFU}{4.8983in}{0.9513in}{0pt}{\Qcb{$G$, $H$ and $H^{\prime
%}$ in case 1.}}{\Qlb{ccircf12}}{ccircf12.ps}%
%width 4.8983in;  height 0.9513in;  depth 0pt;  original-width 8.4968in;
%original-height 11.0056in;  cropleft "0.1569";  croptop "0.8901";
%cropright "0.9212";  cropbottom "0.7781";
%filename 'ccircf12.ps';file-properties "XNPEU";}}}%
%BeginExpansion
\begin{figure}
[ptb]
\begin{center}
\includegraphics[
trim=1.333148in 8.563457in 0.669548in 1.209515in,
height=0.9513in,
width=4.8983in
]%
{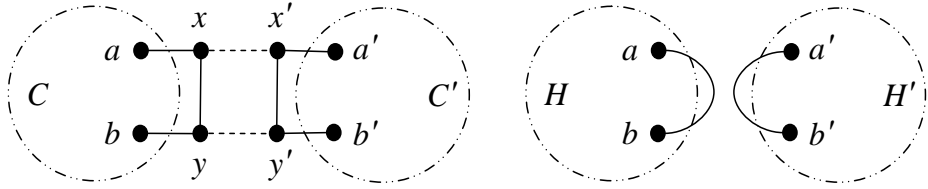}%
\caption{$G$, $H$ and $H^{\prime}$ in case 1.}%
\label{ccircf12}%
\end{center}
\end{figure}
%EndExpansion

As $G$ does not have two pairs of twin vertices, one of $H,H^{\prime}$ must
have the property that all of its pairs of twin vertices are disrupted in $G$.
We presume that $H$ has this property. The minimality of $G$ implies that $G$
has vertices $c$ and $d$ that are twins of $a$ and $b$ (respectively) in $H$.
If $a$ and $c$ are adjacent then $N_{G}(a)=\{c,d,x\}$ and $N_{G}%
(c)=\{a,b,d\}$. It cannot be that $b$ and $d$ are adjacent too, for if they
were then we would have $N_{G}(b)=\{c,d,y\}$ and $N_{G}(d)=\{a,b,c\}$,
implying that $c$ and $d$ are adjacent twins in $G$. Consequently $b$ and $d$
share a neighbor $e\notin\{a,b,c,d,x,y\}$, and we have $N_{G}(b)=\{c,e,y\}$
and $N_{G}(d)=\{a,c,e\}$. But then $e$ is a cutpoint, separating $\{a,b,c,d\}$
from the rest of $C$. As $G$ has no cutpoint, we conclude that $a$ and $c$ are
not adjacent; similarly, $b$ and $d$ are not adjacent.

It follows that $G$ has vertices $e$ and $f$ such that $N_{G}(a)=\{d,e,x\}$,
$N_{G}(b)=\{c,f,y\}$, $N_{G}(c)=\{b,d,e\}$ and $N_{G}(d)=\{a,c,f\}$. As $e$
and $f$ are of degree 3 in $G$, they must be distinct. If $e$ and $f$ were
adjacent, then $\{c,f\}$ and $\{d,e\}$ would be twin pairs in $G$; hence $e$
and $f$ are not adjacent. Let $g$ and $h$ be the third neighbors of $e$ and
$f$, respectively; $g$ and $h$ must be distinct as $G$ has no cutpoint. The
situation in $G$ is pictured on the left in Figure \ref{ccircf12a}. Let
$G^{\prime}$ be the smaller graph obtained from $G$ as indicated on the right
in\ Figure \ref{ccircf12a}. Notice that whether or not $g$ and $h$ are
adjacent, neither can have a twin in either $G$ or $G^{\prime}$; clearly then
all twins of $G^{\prime}$ are also twins of $G$, so $G^{\prime}$ does not have
two pairs of twins. Also, $G^{\prime}$ is a circle graph, because it can be
obtained from $G$ by first deleting $c$ and $d$, and then performing local
complementations and vertex deletions at $a$, $b$, $e$ and $f$. But this
contradicts the minimality of $G$.%

%TCIMACRO{\FRAME{ftbpFU}{4.4901in}{1.0551in}{0pt}{\Qcb{The minimality of $G$ is
%contradicted in case 1.}}{\Qlb{ccircf12a}}{ccircf12a.ps}%
%width 4.4901in;  height 1.0551in;  depth 0pt;  original-width 8.4968in;
%original-height 11.0056in;  cropleft "0.1732";  croptop "0.8901";
%cropright "0.8734";  cropbottom "0.7657";
%filename 'ccircf12a.ps';file-properties "XNPEU";}}}%
%BeginExpansion
\begin{figure}
[ptb]
\begin{center}
\includegraphics[
trim=1.471646in 8.426988in 1.075695in 1.209515in,
height=1.0551in,
width=4.4901in
]%
{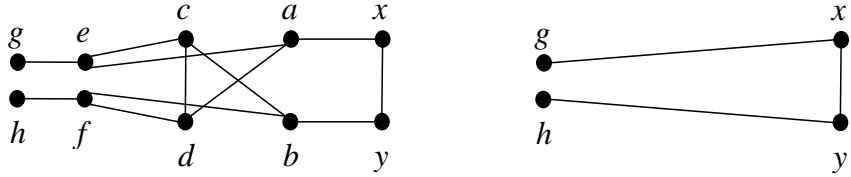}%
\caption{The minimality of $G$ is contradicted in case 1.}%
\label{ccircf12a}%
\end{center}
\end{figure}
%EndExpansion

Case 2. Suppose now that $x$ and $y$ are not neighbors. We denote by $C$ the
component of $G-x-y$ that is connected to $y$ by only one edge. We claim that
we may suppose without loss of generality that $C$ is connected to $x$ by two
edges, as indicated on the left-hand side of Figure \ref{ccircf13}. Suppose
instead that $C$ is connected to $x$ by only one edge; then $x$ has only one
neighbor in $C$, $x^{\prime}$ say. Necessarily $x^{\prime}$ is not the
neighbor of $y$ in $C$; if it were, it would be a cutpoint. It follows that
$\{x^{\prime},y\}$ is a minimal vertex cut in $G$, the induced subgraph
$G[V(C)-x^{\prime}]$ is a component of $G-x^{\prime}-y$, and $x^{\prime}$ is
attached to this component by two edges while $y$ is attached to this
component by only one edge.

Having verified our claim, we presume that two edges connect $x$ to $C$. If
$x$ is adjacent to the neighbor of $y$ in $C$, then $x$ and this neighbor
constitute a 2-element vertex cut of $G$, and we may apply the argument of
case 1. The argument of case 1 applies also if $x$ and $y$ share a neighbor in
$C^{\prime}$, so we may proceed with the assumption that $x$ and $y$ do not
share a neighbor. That is, the vertices denoted $a$, $b$, $c$, $a^{\prime}$,
$b^{\prime}$ and $c^{\prime}$ in Figure \ref{ccircf13} are all distinct.

Let $H$ and $H^{\prime}$ be the two smaller graphs indicated in Figure
\ref{ccircf13}. We claim that they are circle graphs. It is enough to explain
why $H$ is a circle graph. Choose a shortest path from $a^{\prime}$ to one of
$b^{\prime}$, $c^{\prime}$ in $C^{\prime}$. (N.b. Such a path must exist as
$a^{\prime}$ is not a cutpoint of $G$.) Let $H^{\ast}$ be the graph obtained
from $G$ by deleting all vertices of $C^{\prime}$ that do not lie on this
path. Then $H^{\ast}$ consists of the induced subgraph of $G$ with vertex set
$V(C)\cup\{x\}$, and a path of degree-2 vertices connecting $x$ to $c$. Each
of these degree-2 vertices may be removed, by performing a local
complementation and a vertex deletion. The result is $H$.%

%TCIMACRO{\FRAME{ftbpFU}{4.9165in}{0.9989in}{0pt}{\Qcb{$G$, $H$ and $H^{\prime
%}$ in case 2.}}{\Qlb{ccircf13}}{ccircf13.ps}%
%width 4.9165in;  height 0.9989in;  depth 0pt;  original-width 8.4968in;
%original-height 11.0056in;  cropleft "0.1541";  croptop "0.8841";
%cropright "0.9212";  cropbottom "0.7662";
%filename 'ccircf13.ps';file-properties "XNPEU";}}}%
%BeginExpansion
\begin{figure}
[ptb]
\begin{center}
\includegraphics[
trim=1.309357in 8.432490in 0.669548in 1.275549in,
height=0.9989in,
width=4.9165in
]%
{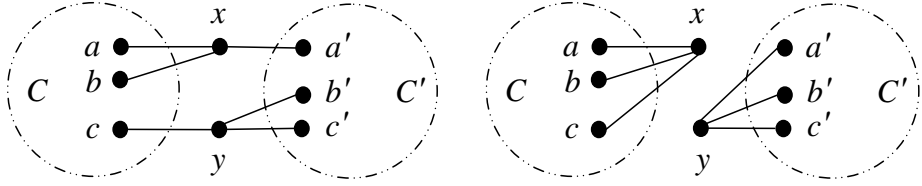}%
\caption{$G$, $H$ and $H^{\prime}$ in case 2.}%
\label{ccircf13}%
\end{center}
\end{figure}
%EndExpansion

As $H$ and $H^{\prime}$ are cubic circle graphs smaller than $G$, each of them
has two disjoint pairs of twin vertices. As $G$ does not have two disjoint
pairs of twin vertices, there must be one of $H$, $H^{\prime}$ for which all
pairs of twins are disrupted in $G$. We presume that no twins of $H$ are twins
in $G$.

The only vertices of $H$ with different neighbors in $G$ and $H$ are $c$ and
$x$. Consequently, it must be that $c$ and $x$ are elements of two disjoint
pairs of twins in $H$, $\{x,x^{\prime}\}$ and $\{c,c^{\prime}\}$. Then
$c^{\prime}$ is a neighbor of $x$, so $c^{\prime}$ is one of $a$, $b$; we may
presume that $c^{\prime}=b$.

Suppose $x^{\prime}$ is an adjacent twin of $x$; then it must be one of $a$,
$b$. (N.b. This situation is not illustrated in a figure.) As the pairs
$\{x,x^{\prime}\}$ and $\{c,c^{\prime}\}$ are disjoint, it must be that
$x^{\prime}=a$. Then $b$ and $c$ are both neighbors of $a$. If $b$ and $c$
were adjacent, it would follow that $a$ and $b$ are adjacent twins in $G$,
contrary to the hypothesis that no twin vertices of $H$ are twin vertices of
$G$. We conclude that $b$ and $c$ are nonadjacent twins in $H$. Consequently
there is a vertex $z$ of $C$ such that $N_{H}(b)=N_{H}(c)=\{a,x,z\}$. This
cannot happen, though, because every path from $x$ to the third neighbor of
$z$ would pass through $z$, i.e., $z$ would be a cutpoint of $G$.

Suppose now that $x^{\prime}$ is a nonadjacent twin of $x$ in $H$; then
$N_{G}(x^{\prime})=N_{H}(x)=\{a,b,c\}$. If $a$ and $b$ were neighbors they
would be adjacent twins in $G$, contrary to hypothesis; so $a$ and $b$ are not
neighbors. The situation in $G$ is pictured on the left-hand side of Figure
\ref{ccircf11}.%

%TCIMACRO{\FRAME{ftbpFU}{4.8931in}{1.1502in}{0pt}{\Qcb{The minimality of $G$ is
%contradicted in case 2.}}{\Qlb{ccircf11}}{ccircf11.ps}%
%width 4.8931in;  height 1.1502in;  depth 0pt;  original-width 8.4968in;
%original-height 11.0056in;  cropleft "0.1732";  croptop "0.8899";
%cropright "0.9369";  cropbottom "0.7538";
%filename 'ccircf11.ps';file-properties "XNPEU";}}}%
%BeginExpansion
\begin{figure}
[ptb]
\begin{center}
\includegraphics[
trim=1.471646in 8.296021in 0.536148in 1.211716in,
height=1.1502in,
width=4.8931in
]%
{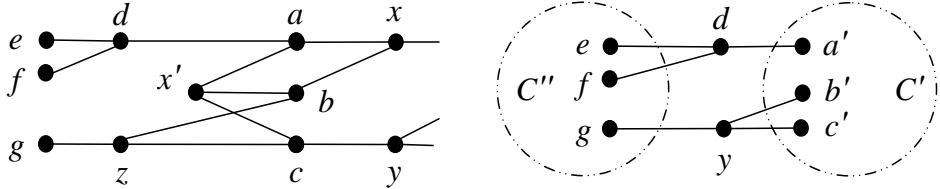}%
\caption{The minimality of $G$ is contradicted in case 2.}%
\label{ccircf11}%
\end{center}
\end{figure}
%EndExpansion

Consider the graph $G^{\prime}$ indicated on the right-hand side of Figure
\ref{ccircf11}. $G^{\prime}$ is a circle graph because we can obtain it from
$G$ by first deleting $b$ and $x^{\prime}$, and then performing local
complementations and vertex deletions at $a$, $c$, $x$ and $z$. As $G^{\prime
}$ is smaller than $G$, it has two disjoint pairs of twin vertices. But
clearly every pair of twins in $G^{\prime}$ yields a pair of twins in $G$, and
this contradicts the minimality of $G$.

As we have reached contradictions in both cases, we conclude that a minimal
counterexample to Theorem \ref{sharp} must be 3-connected.

\subsection{ No counterexample is 3-connected}

The argument in this subsection is quite different from the earlier arguments,
as it is focused on the properties of double occurrence words. We begin with
the following special case of Bouchet's theorem \cite{Bec} that two double
occurrence words with the same prime interlacement graph must be cyclically equivalent.

\begin{proposition}
\label{unique1} \cite{Bec} Suppose $c>4$ and $C_{c}$ is the cycle graph with
vertices $v_{1},...,v_{c}$ (in order). Then up to cyclic equivalence, the only
double occurrence word with interlacement graph $C_{c}$ is this:%
\begin{equation}
v_{1}v_{c}v_{2}v_{1}v_{3}v_{2}v_{4}...v_{c-1}v_{c-2}v_{c}v_{c-1}
\label{cycleword}%
\end{equation}

\end{proposition}

Although $C_{3}$ and $C_{4}$ are not prime, a version of Proposition
\ref{unique1} applies to them too.

\begin{proposition}
\label{unique2}Suppose $c\in\{3,4\}$ and $W$ is a double occurrence word with
$\mathcal{I}(W)=C_{c}$. Then the vertices of $C_{c}$ may be indexed in such a
way that $v_{1},...,v_{c}$ appear in this order on the cycle, and $W$ is
cyclically equivalent to $v_{1}v_{3}v_{2}v_{1}v_{3}v_{2}$ (if $c=3$) or
$v_{1}v_{4}v_{2}v_{1}v_{3}v_{2}v_{4}v_{3}$ (if $c=4$).
\end{proposition}

\begin{proof}
Suppose first that $W=w_{1}w_{2}w_{3}w_{4}w_{5}w_{6}$ has $\mathcal{I}%
(W)=C_{3}$. If $w_{1}$, $w_{2}$ and $w_{3}$ are not distinct letters then one
of them has degree $\leq1$ in the interlacement graph, a contradiction. If
$w_{4}$ is not the same letter as $w_{1}$ then the degree of $w_{4}$ in the
interlacement graph is $\leq1$; hence $w_{1}w_{2}w_{3}w_{4}=v_{1}v_{2}%
v_{3}v_{1}$. If $w_{5}$ is $v_{3}$ then $v_{3}$ does not neighbor $v_{2}$ in
$\mathcal{I}(W)$, a contradiction; hence $w_{1}w_{2}w_{3}w_{4}w_{5}=v_{1}%
v_{2}v_{3}v_{1}v_{2}$ and of course $w_{6}=v_{3}$ as that is the only
remaining possibility.

Now suppose $W=w_{1}w_{2}w_{3}w_{4}w_{5}w_{6}w_{7}w_{8}$ has $\mathcal{I}%
(W)=C_{4}$. Let $v_{1}$ denote the vertex corresponding to $w_{1}$. Suppose
the vertex corresponding to $w_{2}$ is adjacent to $v_{1}$ in $C_{4}$, and
denote the corresponding vertex $v_{4}$. The vertex corresponding to $w_{3}$
cannot be either $v_{1}$ or $v_{4}$, for if it were then its degree in
$\mathcal{I}(W)$ would be $<2$. Call this vertex $v_{2}$. If $v_{2}$ is not
adjacent to $v_{1}$ then $W$ must be of the form $v_{1}v_{4}v_{2}%
...v_{2}...v_{1}...$ But then it is impossible to place the second appearance
of $v_{4}$ so as to interlace both $v_{1}$ and $v_{2}$. Consequently $v_{2}$
is adjacent to $v_{1}$. The vertex corresponding to $w_{4}$ cannot be either
$v_{2}$ or $v_{4}$, for if it were then its degree would be $<2$. If it is the
remaining vertex $v_{3}$ then as $v_{1}$ and $v_{3}$ are not adjacent in
$C_{4}$, $W$ is of the form $v_{1}v_{4}v_{2}v_{3}...v_{3}...v_{1}...$ and it
is impossible to locate the second appearance of $v_{2}$ so as to interlace
both $v_{1}$ and $v_{3}$. Consequently $w_{4}$ is the second appearance of
$v_{1}$, and $W$ is of the form $v_{1}v_{4}v_{2}v_{1}w_{5}w_{6}w_{7}w_{8}$.
The remaining vertex $v_{3}$ is of degree 2 in $\mathcal{I}(W)$, and this can
only happen if $w_{5}$ and $w_{8}$ are its two appearances. Necessarily then
$w_{6}$ and $w_{7}$ are $v_{2}$ and $v_{4}$; as they are not neighbors in
$\mathcal{I}(W)$, $W$ must be of the form $v_{1}v_{4}v_{2}v_{1}v_{3}v_{2}%
v_{4}v_{3}$.

It remains to consider the possibility that $W=w_{1}w_{2}w_{3}w_{4}w_{5}%
w_{6}w_{7}w_{8}$ has $\mathcal{I}(W)=C_{4}$, $v_{1}$ denotes the vertex
corresponding to $w_{1}$, and the vertex corresponding to $w_{2}$ is not
adjacent to $v_{1}$ in $C_{4}$. In this case we observe that the vertex
corresponding to $w_{8}$ cannot be the vertex corresponding to either $w_{1}$
or $w_{2}$; if it were its degree in $\mathcal{I}(W)$ would be $<2$.
Consequently $w_{8}$ is a neighbor of $v_{1}$, and we may apply the argument
of the preceding paragraph to the cyclically equivalent word $w_{1}w_{8}%
w_{7}w_{6}w_{5}w_{4}w_{3}w_{2}$.
\end{proof}

Note that a cubic graph cannot be a forest, as it has no vertex of degree 1.
Consequently every cubic graph has circuits. Suppose $W$ is a double
occurrence word whose interlacement graph is $\mathcal{I}(W)=G$, a cubic
circle graph. Let $C$ be a circuit in $G$, of minimal length $c\geq3$. If we
remove from $W$ all occurrences of vertices that do not appear on $C$, we must
obtain a subword $W^{\prime}$ whose interlacement graph is $C$.\ (Note that
the minimality of $c$ guarantees that $C$ has no chord in $G$.) Propositions
\ref{unique1} and \ref{unique2} tell us that we may index the vertices that
appear on $C$ as $v_{1}$, ..., $v_{c}$, in order of their appearance, and
$W^{\prime}$ will be cyclically equivalent to (\ref{cycleword}). Consequently,
we may assume that $W$ is of the form%

\begin{equation}
v_{1}W_{1}v_{c}W_{2}v_{2}W_{3}v_{1}W_{4}v_{3}W_{5}v_{2}W_{6}v_{4}W_{7}%
v_{3}...v_{c-2}W_{2c-2}v_{c}W_{2c-1}v_{c-1}W_{2c}. \label{mainword}%
\end{equation}

When we reference this description of $W$ we will consider the index of
$W_{i}$ modulo $2c$, so that $W_{0}=W_{2c}$, $W_{1}=W_{2c+1}$, etc.

Observe that if $v$ appears in two non-consecutive $W_{i}$ then $v$\ must
neighbor at least two of $v_{1},...,v_{c}$. For instance, a vertex that
appears once in $W_{2}$ and once in $W_{6}$ neighbors $v_{1}$ and $v_{3}$. On
the other hand, a vertex that appears once in each of two consecutive $W_{i}$
neighbors exactly one of $v_{1},...,v_{c}$, and a vertex that appears twice in
the same $W_{i}$ does not neighbor any of $v_{1},...,v_{c}$.

\begin{proposition}
Let $G$ be a cubic circle graph that does not have two pairs of twin vertices.
Then $G$ is not 3-connected.
\end{proposition}

\begin{proof}
Suppose instead that $G$ is a 3-connected, cubic circle graph, which does not
have two pairs of twin vertices. We may presume that $G$ is of the smallest
possible order for such a graph. Let~$W$ be a double occurrence word of the
form (\ref{mainword}), with $\mathcal{I}(W)=G$.

\textbf{Claim} If $v\notin\{v_{1},...,v_{c}\}$ then either $v$ appears twice
in the same subword $W_{i}$, or $v$ appears once in each of two consecutive
subwords $W_{i}$ and $W_{i+1}$.

Suppose the claim is false; then some $v\notin\{v_{1},...,v_{c}\}$ neighbors
more than one of $v_{1},...,v_{c}$. Cyclically permuting indices if necessary
we may presume that $v$ neighbors $v_{1}$ and some other $v_{j}$.

If $v$ neighbors $v_{j}$, $j>3$, then $v_{1},v,v_{j},...,v_{c}$ is a closed
walk in $G$, of length $<c$. \ This contradicts the choice of $c$. If $v$
neighbors $v_{2}$ then $\{v,v_{1},v_{2}\}$ is a circuit in $G$, so $c=3$. But
then $v_{1}$ and $v_{2}$ both neighbor\ $v$, and also $v_{1}$ and $v_{2}$ both
neighbor $v_{3}$. As $v_{1}$ and $v_{2}$ are neighbors of degree 3, it follows
that $\{v,v_{3}\}$ is a vertex cut, which separates $\{v_{1},v_{2}\}$ from the
rest of $G$. As $G$ is 3-connected, there must be no \textquotedblleft rest of
$G$\textquotedblright\ -- i.e., $v$, $v_{1}$, $v_{2}$ and $v_{3}$ are all the
vertices $G$ has. But then $G$ is a 4-clique, contradicting the hypothesis
that $G$ does not have two disjoint pairs of twin vertices.

Suppose $v$ neighbors $v_{1}$ and $v_{3}$. Then $v,v_{1},v_{2},v_{3}$ is a
closed walk in $G$, so $c\leq4$. If $c=3$ then $v_{1},v_{2},v_{3}$ and
$v_{1},v,v_{3}$ are both 3-circuits in $G$, so $\{v,v_{2}\}$ is a vertex cut
that separates $\{v_{1},v_{3}\}$ from the rest of $G$. Again, this contradicts
either the hypothesis that $G$ is 3-connected or the hypothesis that $G$ does
not have two disjoint pairs of twin vertices.

We conclude that $c=4$. Then $v_{1},v_{2},v_{3},v_{4}$ is a circuit of $G$, so
$v_{1}$ and $v_{3}$ both neighbor $v$, $v_{2}$ and $v_{4}$. Consequently,
$v_{1}$ and $v_{3}$ are nonadjacent twins. No two of $v$, $v_{2}$ and $v_{4}$
may share another neighbor; for if they were to share another neighbor then
they would be twins, and $G$ does not have two pairs of twins. Let $x$,
$x_{2}$ and $x_{4}$ be the third neighbors of $v$, $v_{2}$ and $v_{4}$
respectively, as on the left-hand side of Figure \ref{ccircf2}.%
%TCIMACRO{\FRAME{ftbpFU}{2.7899in}{1.5515in}{0pt}{\Qcb{The last case of the
%claim.}}{\Qlb{ccircf2}}{ccircf2.ps}{\special{ language "Scientific Word";
%type "GRAPHIC";  maintain-aspect-ratio TRUE;  display "USEDEF";
%valid_file "F";  width 2.7899in;  height 1.5515in;  depth 0pt;
%original-width 8.4968in;  original-height 11.0056in;  cropleft "0.1888";
%croptop "0.9025";  cropright "0.6222";  cropbottom "0.7177";
%filename 'ccircf2.ps';file-properties "XNPEU";}}}%
%BeginExpansion
\begin{figure}
[ptb]
\begin{center}
\includegraphics[
trim=1.604196in 7.898719in 3.210091in 1.073046in,
height=1.5515in,
width=2.7899in
]%
{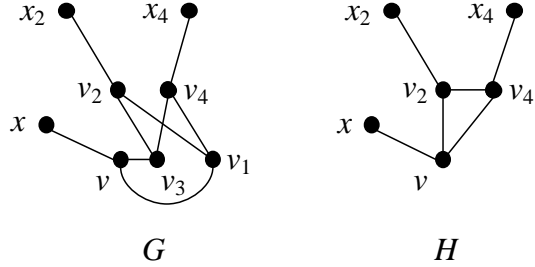}%
\caption{The last case of the claim.}%
\label{ccircf2}%
\end{center}
\end{figure}
%EndExpansion

Let $H$ be the graph obtained by replacing the pictured portion of $G$ with
the smaller subgraph indicated on the right-hand side of Figure \ref{ccircf2}.
Then $H$ is obtained from $G$ in three steps: delete $v_{1}$, take the local
complement with respect to $v_{3}$, and delete $v_{3}$. Consequently $H$ is a
circle graph. Clearly $H$ is also cubic and 3-connected.

The minimality of $G$ assures us that $H$ has two pairs of twin vertices. The
six indicated vertices of $H$ are all distinct, so no two of them are twins
(as their neighborhoods are distinct). Any twin vertices outside the pictured
portion of $H$ are also twins in $G$, though, and this implies that $G$ has
three pairs of twin vertices, contradicting the hypothesis that it does not
even have two pairs. The contradiction verifies the claim.

As $G$ is cubic and $C$ is a chordless circuit, each $v_{j}$ has precisely one
neighbor $u_{j}$ that does not appear on $C$. The claim tells us that $u_{j}$
appears in two consecutive subwords $W_{i}$ and $W_{i+1}$. As $u_{j}$
neighbors $v_{j}$, $W_{i}v_{j}W_{i+1}v_{k}$ must be a subword of $W$, for some
$k$. Notice that as $W_{i}$ and $W_{i+1}$ do not mention any of $v_{1}$, ...,
$v_{c}$, $u_{j}$ does not neighbor any of $v_{1}$, ..., $v_{c}$ other than
$v_{j}$; this holds for every $j$, so $u_{1}$, ..., $u_{c}$ are pairwise distinct.

If $W_{i+1}$ contains any vertex $v$ other than $u_{j}$ or $u_{k}$, then $v$
appears twice in $W_{i+1}$, so every walk from $v$ to $v_{j}$ in $G$ must pass
through $u_{j}$ or $u_{k}$. As $G$ is 3-connected, it follows that there is no
such $v$. The same argument applies to every $W_{i}$, as the claim implies
that no more than two of $u_{1}$, ..., $u_{c}$ appear in any one $W_{i}$.\ We
conclude that $V(G)=\{u_{1}$, ..., $u_{c}$, $v_{1}$, ..., $v_{c}\}$.

After reversing or cyclically permuting $W$ if necessary, we may presume that
$u_{1}$ appears in $W_{3}$ and $W_{4}$. The degree of $u_{1}$ is 3, so the
subword $v_{2}W_{3}v_{1}W_{4}v_{3}$ of $W$ must be $v_{2}u_{1}u_{2}v_{1}%
u_{3}u_{1}v_{3}$. The degree of $u_{2}$ is also 3, so the subword $v_{c}%
W_{2}v_{2}W_{3}v_{1}W_{4}v_{3}$ must be $v_{c}u_{2}u_{c}v_{2}u_{1}u_{2}%
v_{1}u_{3}u_{1}v_{3}$. Notice that $u_{2}$ cannot appear in $W_{5}$, as it
appears in $W_{2}$ and $W_{3}$; consequently the subword $W_{4}v_{3}W_{5}%
v_{2}$ of $W$ must be $u_{3}u_{1}v_{3}u_{3}v_{2}$. But then the degree of
$u_{3}$ is 2, contradicting the hypothesis that $G$ is 3-regular.
\end{proof}

\section{Corollary \ref{connect}}

In this section we derive Corollary \ref{connect} from Theorem \ref{sharp}.

Suppose $G$ is a cubic circle graph, which is 3-connected. If $G$ has a pair
of adjacent twins, $v$ and $w$, then they share two neighbors, $x$ and $y$. If
$z$ is any other vertex of $G$ then every path from $v$ to $z$ in $G$ must
pass through $x$ or $y$; as $G$ is 3-connected, this cannot be the case.
Consequently $G$ has no other vertex, i.e., $V(G)=\{v,w,x,y\}$. As $G$ is a
cubic graph, $G\cong K_{4}$.

If $G$ has no pair of adjacent twins then Theorem \ref{sharp} tells us that
$G$ has two disjoint pairs of nonadjacent twins, $\{v,v^{\prime}\}$ and
$\{w,w^{\prime}\}$. Suppose $v$ and $w$ are neighbors; then $v$ and
$v^{\prime}$ are neighbors of $w$ and $w^{\prime}$. Let the third neighbor of
$v$ and $v^{\prime}$ be $x$, and let the third neighbor of $w$ and $w^{\prime
}$ be $y$. If $V(G)\not =\{v,v^{\prime},w,w^{\prime},x,y\}$, then every path
from one of $v$, $v^{\prime}$, $w$, $w^{\prime}$ to a vertex outside
$\{v,v^{\prime},w,w^{\prime},x,y\}$ passes through $x$ or $y$. As $G$ is
3-connected, this cannot be the case; we conclude that $V(G)=\{v,v^{\prime
},w,w^{\prime},x,y\}$ and hence $G\cong K_{3,3}$.

Suppose now that $v$ and $w$ are not neighbors; we claim that this is
impossible. The graph $G$ is 3-connected, so Menger's theorem tells us that
there are three internally vertex-disjoint paths from $v$ to $w$. We may
presume that no edge of $G$ connects two non-consecutive vertices of any of
the three paths; for if there is such an edge we may use it to shorten that
path. Let $H$ be the full subgraph of $G$ induced by the vertices on these
three paths, including $v$ and $w$. Then $H$ is a circle graph. Let the three
paths be $v$, $x_{1}$, ..., $x_{p}$, $w$; $v$, $y_{1}$, ..., $y_{q}$, $w$; and
$v$, $z_{1}$, ..., $z_{r}$, $w$. $H$ is pictured on the left in Figure
\ref{ccircf15}.

Note that as $G$ is 3-regular, it must be that $x_{1}\neq x_{p}$; for $x_{1}$
is adjacent to both $v$ and $v^{\prime}$, and $x_{p}$ is adjacent to both $w$
and $w^{\prime}$. The same argument tells us that $y_{1}\neq y_{q}$ and
$z_{1}\neq z_{r}$. As $v^{\prime}$ and $w^{\prime}$ are not vertices of $H$,
$x_{1}$, $y_{1}$, $z_{1}$, $x_{p}$, $y_{q}$ and $z_{r}$ are all of degree 2 in
$H$. Consequently $x_{1}$, $y_{1}$, $z_{1}$, $x_{p}$, $y_{q}$ and $z_{r}$ are
all of degree 3 in the graph $H^{\prime}$ obtained from $H$ by performing
local complementations and vertex deletions at $v$ and $w$. If any vertex of
$H^{\prime}$ is of degree 2, we may remove it by performing a local
complementation and then a vertex deletion. The resulting graph $H^{\prime
\prime}$ differs from $H^{\prime}$ in that some indices may not appear on the
paths $x_{1}$, ..., $x_{p}$; $y_{1}$, ..., $y_{q}$; and $z_{1}$, ..., $z_{r}$.
But each path will still involve at least two distinct vertices, so no two
vertices of $H^{\prime\prime}$ will be twins.

As $H^{\prime\prime}$ is a cubic circle graph, Theorem \ref{sharp} verifies
the claim that this situation is impossible.%
%TCIMACRO{\FRAME{ftbpFU}{3.8891in}{1.2522in}{0pt}{\Qcb{The graphs $H$ (on the
%left) and $H^{\prime}$, $H^{\prime\prime}$ (on the right).}}{\Qlb{ccircf15}%
%}{ccircf15.ps}{\special{ language "Scientific Word";  type "GRAPHIC";
%width 3.8891in;  height 1.2522in;  depth 0pt;  original-width 8.4968in;
%original-height 11.0056in;  cropleft "0.2830";  croptop "0.9020";
%cropright "0.8883";  cropbottom "0.7537";
%filename 'ccircf15.ps';file-properties "XNPEU";}}}%
%BeginExpansion
\begin{figure}
[ptb]
\begin{center}
\includegraphics[
trim=2.404594in 8.294921in 0.949093in 1.078549in,
height=1.2522in,
width=3.8891in
]%
{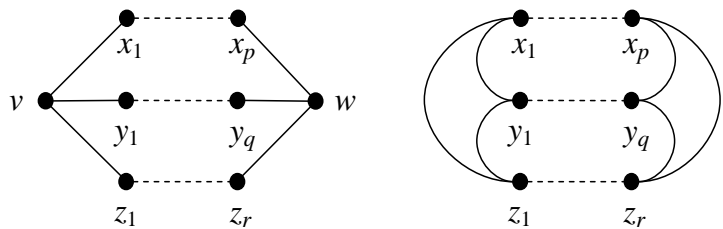}%
\caption{The graphs $H$ (on the left) and $H^{\prime}$, $H^{\prime\prime}$ (on
the right).}%
\label{ccircf15}%
\end{center}
\end{figure}
%EndExpansion

\end{document}